\documentclass{amsart}    
\usepackage{amsmath} 
\usepackage{paralist}
\usepackage{cancel} 
\usepackage{graphics} 
\usepackage{epsfig} 
\usepackage{graphicx}  
\usepackage{epstopdf} 
\usepackage{mathrsfs,dsfont}
\usepackage[colorlinks=true]{hyperref}
\hypersetup{urlcolor=blue, citecolor=red}

\newtheorem{Definition}{Definition}[section] 
\newtheorem{Proposition}{Proposition}[section]
\newtheorem{Lemma}{Lemma}[section]
\newtheorem{Theorem}{Theorem}[section]
\newtheorem{Corollary}{Corollary}[section]
\newtheorem{Remark}{Remark}[section]

\def \vu{\vec{u}}

\def \vn{\vec{\nabla}}

\def \P{\mathbb{P}}

\def \R{\mathbb{R}}

\def \Rt{\mathbb{R}^3}

\def \ds{\displaystyle}

\def \bl{\textcolor{black}}

\usepackage{amssymb}

	\title[\bf Very weak suitable solutions for the Navier-Stokes equations] %
{A remark on very weak suitable solutions and Leray solutions of the Navier-Stokes equations}  
\author[Oscar Jarr\'in]{}


\email{oscar.jarrin@udla.edu.ec}

\thanks{$^*$Corresponding author:  Oscar Jarr\'in}

\begin{document}
	\maketitle

\begin{center}
        \begin{minipage}{5cm}
	\centerline{\scshape Oscar Jarr\'in$^*$}
	\medskip
	{\footnotesize
		\centerline{Escuela de Ciencias Físicas y Matemáticas}
		\centerline{Universidad de Las Américas}
		\centerline{V\'ia a Nay\'on, C.P.170124, Quito, Ecuador}
	} 
\end{minipage}
\end{center}	
	
\bigskip
\begin{abstract} We introduce the notion of very weak suitable solutions for the Navier--Stokes equations. Here, the velocity and the pressure satisfy minimal conditions that make sense of the local energy inequality in the distributional setting.

A well-known but still relevant question is to find sufficient conditions ensuring that very-weak solutions are in fact Leray solutions. Exploiting the \emph{local} energy inequality within the general framework of \emph{local} Morrey spaces, we establish such conditions.

Local Morrey spaces provide a general framework that contains other useful functional settings in the theoretical analysis of the Navier--Stokes equations, such as Lebesgue, Lorentz, homogeneous Morrey, and parabolic Morrey spaces.
	
As a by-product, we also derive some sufficient conditions to study the uniqueness and regularity of the resulting Leray solutions.
\end{abstract}

\section{Introduction}
\subsection{Setting}In this note, we consider the incompressible Navier--Stokes equations posed on the whole three-dimensional space.  For a solenoidal velocity $\vu(t,x)\in \Rt$ and the pressure $P(t,x)\in \R$, with $t>0$ and $x\in \Rt$, these equations take the form

\begin{equation}\label{NS}
\begin{cases}\vspace{2mm}
\partial_t \vu - \Delta \vu + \text{div} (\vu \otimes \vu) + \vn P =0, \quad \text{div}(\vu)=0, \\
\vu(0,\cdot)=\vu_0,  \quad \text{div}(\vu_0)=0,
\end{cases}
\end{equation}
where $\vu_0(x)\in \Rt$ denotes the initial data. With a slight loss of generality, all physical constants are set equal to one as they do not play any substantial role in our study. 

\medskip

We begin by briefly recalling some well-known definitions of very weak solutions and Leray solutions to equations (\ref{NS}).

\begin{Definition}[Very weak solution: Definition $6.3$ of \cite{PLe}]\label{Def-Very-Weak-Sol}
	Let $0<T<+\infty$ be fixed. We say that $\big(\vu, P\big)$ is a very weak solution of the Navier--Stokes equations (\ref{NS}) on $[0,T]\times \Rt$ if the following statements hold.
	\begin{enumerate}
		\item We have $\ds{\vu_0\in \mathcal{D}'(\Rt)}$, $\ds{\vu  \in L^2_{loc}\big([0,T]\times \Rt\big)}$, and $P \in \mathcal{D}' \big( [0,T]\times \Rt \big)$.
		
		\medskip
		
		\item The map $t\mapsto \vu(t,\cdot)$ is continuous from $]0,T]$ into $\mathcal{D}'(\Rt)$, and we have $\ds{\lim_{t\to 0^{+}}\vu(t,\cdot)= \vu_0}$ in $\mathcal{D}'(\Rt)$. 
		
		\medskip
		
		\item The pair $(\vu, P)$ satisfies equation (\ref{NS}) in the distributional sense, \emph{i.e.}, for any $\vec{\varphi}\in \mathcal{D}\big( [0,T]\times \Rt \big)$ with ${\rm div}(\vec{\varphi})=0$, we have  
		\begin{equation}\label{Distributional-Sense}
		\big\langle  \partial_t \vu - \Delta \vu + {\rm div} (\vu \otimes \vu) , \, \vec{\varphi} \big\rangle_{\mathcal{D}' \times \mathcal{D}}.
		\end{equation}
	\end{enumerate}		
\end{Definition}

In the identity above, note that the third term on the left-hand side is well defined as long as the velocity $\vu$ is locally square integrable in time and space. This is the only assumption on $\vu$. Moreover, the pressure $P$ is implicitly defined by equation (\ref{Distributional-Sense}) as follows: defining $\vec{\Phi}:=-\partial_t \vu + \Delta \vu - \text{div} (\vu \otimes \vu) \in \mathcal{D}'\big( [0,T]\times \Rt\big)$, we have $\vec{\nabla}\wedge \vec{\Phi}=0$, and by \cite[Lemma 6.3]{PLe} there exists $P \in \mathcal{D}'\big( [0,T]\times \Rt\big)$ such that $\vec{\Phi}=\vec{\nabla}P$. Additionally, note that we allow the initial velocity $\vu_0$ to be a general object in $\mathcal{D}'(\Rt)$.

\medskip

On the other hand, when the initial datum satisfies $\vu_0 \in L^2(\Rt)$, since the classical results of J. Leray \cite{Leray} and E. Hopf \cite{Hopf}, the existence of global-in-time weak solutions in the natural energy space is well known. These solutions are the so-called Leray solutions, and they satisfy the following properties.

\begin{Definition}[Leray Solution: Definition $12.1$ of \cite{PLe}]\label{Def-Leray-Solution}
	We say that $(\vu, P)$ is a Leray solution of the Navier-Stokes equations (\ref{NS}) if the following statements hold:
	\begin{enumerate}
		\item For any $T>0$ we have $\vu \in L^\infty\big( [0,T], L^2(\Rt)\big)\cap L^2\big( [0,T], \dot{H}^1(\Rt)\big)$, $P \in L^2\big( [0,T], \dot{H}^{-\frac{1}{2}}(\Rt)\big)$, and $(\vu, P)$ satisfy the equation (\ref{NS}) in the distributional sense (\ref{Distributional-Sense}).

		\medskip
		
		\item For every $0<t\leq T$, the following global energy inequality holds:
		\begin{equation}\label{Energy-Inequality}
		\begin{split}
		\| \vu(t,\cdot)\|^{2}_{L^2}
		+\int_{0}^{t}
		\| \vu(s,\cdot)\|^2_{\dot{H}^1} ds  \leq \| \vu_0 \|^2_{L^2}. 
		\end{split}
		\end{equation}
		
	\end{enumerate}
\end{Definition}

In this context, a well-known but still relevant question in Navier--Stokes theory is to determine sufficient conditions under which a very weak solution of equations (\ref{NS}), in the sense of Definition \ref{Def-Very-Weak-Sol}, becomes a Leray solution in the sense of Definition \ref{Def-Leray-Solution}.

\medskip

Besides the natural assumption on the initial datum that $\vu_0 \in L^2(\Rt)$, these conditions consist of two main hypotheses. On the one hand, for the velocity $\vu \in L^2_{\rm loc}\big( [0,T]\times \Rt \big)$ and for a suitable functional space $E \subset L^2_{\rm loc}\big( [0,T]\times \Rt \big)$, we assume that $\vu \in E$. On the other hand, we also need a precise hypothesis on the limit
\begin{equation}\label{Limit}
\lim_{t\to 0^{+}}\vu(t,\cdot)=\vu_0.
\end{equation}

In order to find some candidates for the space $E$, we recall that  from the works \cite{Ladyzhenskaya,Prodi,Serrin} it is well-known that \emph{Leray solutions} of equations (\ref{NS})  \emph{satisfying the Ladyzhenskaya-Prodi-Serrin criterion} :
\begin{equation}\label{Ladyzhenskaya - Prodi - Serrin criterion}
\vu \in L^p\big( [0,T], L^q(\Rt) \big), \quad \text{with}\quad \frac{2}{p}+\frac{3}{q}=1, \quad 3<q<+\infty,
\end{equation}
are regular, satisfy the global energy equality, \emph{i.e.}, the expression given in (\ref{Energy-Inequality}) with a strict equality, and are unique. 

\medskip

Interestingly, the space $L^p_t L^q_x$ has also been used to study the passage from very weak solutions to Leray solutions of the Navier--Stokes equations. In the following lines, we provide a non-exhaustive list of previous works.

\medskip  

To the best of our knowledge, the property that a very weak solution becomes a Leray solution  was first proven in \cite[Theorem $2$]{Foias}, where $\vu \in L^p\big( [0,T], L^q(\Rt) \big)\cap L^2\big( [0,T], H^1(\Rt)\big)$ with $\frac{2}{p}+\frac{3}{q}<1$, and the limit (\ref{Limit}) is assumed in the weak topology of $L^2(\Rt)$. Note that here $p$ and $q$ satisfy a subcritical case of condition (\ref{Ladyzhenskaya - Prodi - Serrin criterion}).

\medskip

Later, regarding the critical case of condition (\ref{Ladyzhenskaya - Prodi - Serrin criterion}), in \cite[Theorem $5.3$]{Fabes}, this property was proven under the following assumptions: $\vu \in L^p\big( [0,T], L^q(\Rt) \big)$ (with $p$ and $q$ satisfying (\ref{Ladyzhenskaya - Prodi - Serrin criterion})), together with the hypothesis that $\vu$ also solves the mild formulation\footnote{It satisfies: $\vu(t,\cdot)= e^{\Delta t} \vu_0 - \int_{0}^{t} e^{\Delta(t-s)}\P \text{div}(\vu \otimes \vu)(s,\cdot) ds$,  where $\P$ denotes the Leray projector} of the Navier-Stokes equations.  Additionally, for $3<p<+\infty$, we have $\vu_0 \in L^2 \cap L^p(\Rt)$, and the limit (\ref{Limit}) is assumed in the weak topology of $L^2(\Rt)$.

\medskip

Thereafter, in \cite[Proposition $1$]{Giga}, this property was proven by assuming that  $\vu\in  L^p\big( [0,T], L^q(\Rt)\big)$, where $p,q$ satisfy (\ref{Ladyzhenskaya - Prodi - Serrin criterion}), together with the assumptions that
\begin{equation*}
t^{\frac{1}{p}}\vu \in  L^\infty \cap \mathcal{C} \big( [0,T], L^q(\Rt)\big) \quad \text{and} \quad \lim_{t \to 0^{+}} t^{\frac{1}{p}}\| \vu(t,\cdot)\|_{L^q}=0.
\end{equation*}
Moreover, we have $\vu_0 \in L^2(\Rt) \cap L^3(\Rt)$, and (\ref{Limit}) is also assumed in the weak topology of $L^2(\Rt)$.

\medskip

More recently, in \cite[Theorem $1$]{Galdi}, this property was proven under the following assumptions. For a small parameter $\delta>0$, we have $\vu \in  L^p\big( [\delta,T], L^q(\Rt) \big)$, where $p$ and $q$ are given by (\ref{Ladyzhenskaya - Prodi - Serrin criterion}), or $\vu \in  \mathcal{C}\big([\delta, T], L^3(\Rt)\big)$. Additionally, the limit (\ref{Limit}) is assumed in the weak topology of $L^2(\Rt)$, where the initial datum only belongs to $L^2(\Rt)$, and it is also proven that the limit holds in the strong topology.

\medskip

Finally, this property was also proven in the supercritical case of condition (\ref{Ladyzhenskaya - Prodi - Serrin criterion}). In \cite[Lemma $2.1$]{Ding},  we have $\vu \in L^4 \big( [\delta,T], L^4(\Rt)\big)$ and the limit (\ref{Limit}) is assumed in the strong topology of $L^2(\Rt)$. Specifically, this last result is a key ingredient in proving a more general theorem (see \cite[Theorem $1.1$]{Ding}), where this property is proven in $\vu \in  L^{p_k}\big([\delta,T], L^{q_k}(\Rt)\big)$, with a family of parameters $(p_k, q_k)_{k\geq 0}$ satisfying, among other technical assumptions, that $\frac{2}{r_k}+ \frac{3}{q_k}= 1 + \frac{1}{2^{k+2}}$. On the other hand, it is also worth mentioning the work \cite[Theorem $1.1$]{Galdi-L4} where this property is  proven in the space $L^4\big([0,T],L^4(\Rt) \big)$. Additionally, the energy control (\ref{Energy-Inequality}) is satisfied with an equality sign. 

\medskip

Returning to the \emph{Ladyzhenskaya-Prodi-Serrin} criterion, we know that Leray solutions of equations (\ref{NS}) satisfying $\vu \in L^p_t L^q_x$ (with $p,q$ given by (\ref{Ladyzhenskaya - Prodi - Serrin criterion})) are regular. Specifically, from \cite[Corollary $1.4$]{Galdi} we have that  $\vu \in \mathcal{C}^\infty\big( ]0,T]\times \Rt \big)$ and consequently $P$ is also a smooth function. Therefore, they satisfy the \emph{local energy equality}:
\begin{equation*}
\partial_t |\vu|^2 + 2|\vn \otimes \vu |^2 - \Delta|\vu|^2 + \text{div}\Big( |\vu|^2 \vu  + 2 P \vu\Big)=0.
\end{equation*}

This remark motivates us to introduce the following definition:

\begin{Definition}[Very weak suitable solution]\label{Def-Weak-Suitable-Sol}
	Let $\big(\vu,  P\big)$ be a very-weak solution of the Navier-Stokes equations (\ref{NS}) in the sense of Definition \ref{Def-Very-Weak-Sol}. We will say that $\big(\vu,  P\big)$ is a very weak \emph{suitable} solution if the following statements hold.
	\begin{enumerate}
		\item We have $\vu  \in L^3_{\rm loc}([0,T]\times \Rt)$ and $P \in L^{\frac{3}{2}}_{\rm loc}([0,T]\times \Rt)$.
		
		\medskip
		
		\item The distribution $\mu \in  \mathcal{D}'\big( [0,T]\times \Rt\big)$ given by the expression 
		\begin{equation}\label{Local-Energy}
		\begin{split}
	\mu:=- \partial_t |\vu|^2 - 2|\vn \otimes \vu |^2 + \Delta|\vu|^2 -\, \text{div}\Big( |\vu|^2 \vu  + 2 P \vu\Big),
		\end{split}
		\end{equation}
		defines a non-negative locally finite measure on $[0,T] \times \Rt$. 
	\end{enumerate}			
\end{Definition}

Note that the conditions stated in the first point above are the minimal ones ensuring that the right-hand side in (\ref{Local-Energy}) is well defined in the distributional sense.

\medskip

The non-negativity of the measure $\mu$ given in (\ref{Local-Energy}) is the main feature of this kind of very-weak solutions. This additional property is motivated by the well-known \emph{Cafarelli-Konh-Niremberg} partial regularity theory for the Navier-Stokes equations \cite{Caffarelli}, where a Leray solution (in the sense of Definition \ref{Def-Leray-Solution}) is called \emph{suitable}, provided that the second point above holds. Nevertheless, we emphasize that the definition above does not assume that $\vu \in L^\infty_t L^2_x \cap L^2_t \dot{H}^{1}_x$. Thus, we still work within the framework of very weak solutions given in Definition \ref{Def-Very-Weak-Sol}.

\medskip

Motivated by the previous works on these solutions, it is thus interesting to ask whether very-weak suitable solutions become Leray solutions. The main objective of this note is to study this question. In addition, we show that the additional \emph{suitability} assumption on very-weak solutions allows us to consider considerably more general functional spaces than $L^p_t L^q_x$.

\subsection{Main results} A class of \emph{local Morrey spaces} (see expression (\ref{Local-Morrey-tx}) below for the definition), which, roughly speaking, characterizes the averaged decay properties of functions, has recently attracted attention in the study of the existence of global-in-time weak solutions for the classical Navier--Stokes equations \bl{\cite{BradKu,Fernandez-Lemarie}}, as well as for the coupled system of Magneto-hydrodynamics equations \bl{\cite{Fernandez-Jarrin}}.

\medskip

We thus use these spaces as an interesting and general functional framework to study the property that very weak suitable solutions of the Navier-Stokes equations become Leray solutions.

\medskip 

For $0<T<+\infty$ fixed and for the parameters $\gamma>0$ and $1<p<+\infty$, we define the Banach space  $M^{p,\gamma}_{t,x}([0,T]\times \Rt)$ as the set of functions $\phi \in L^p_{\rm loc}([0,T]\times \Rt)$ which satisfy:
\begin{equation}\label{Local-Morrey-tx}
\| \phi \|_{M^{p,\gamma}_{t,x}}:= \sup_{R\geq 1} \left( \frac{1}{R^\gamma} \int_{0}^{T} \int_{|x|<R} |\phi(t,x)|^p \, dx dt \right)^{\frac{1}{p}}<+\infty. 
\end{equation}

In  expression (\ref{Local-Morrey-tx}), the parameter $\gamma >0$ controls the asymptotic behavior, as $R \to +\infty$, of the local temporal-spatial $L^p$-norm of the function $\phi$.  Later, we describe some functional spaces that are contained within the class of local Morrey spaces, including the Lebesgue, Lorentz, and classical homogeneous Morrey spaces.

\medskip

Now, we are able to state our main result. 
\begin{Theorem}\label{Th1} Let $(\vu,  P)$ be a  very  weak  suitable solution of the Navier-Stokes equations  (\ref{NS})  in the sense of Definition \ref{Def-Weak-Suitable-Sol}. 
	
\medskip	

For the velocity $\vu$, assume that the following statements hold:
\begin{enumerate}
\item For the initial datum we have  $\vu_0 \in L^2(\Rt)$. Additionally, for any compact set $\Omega\subset \Rt$, the map $t \mapsto \vu(t,\cdot)$ in continuous in the weak topology of $L^2(\Omega)$ in the interval $0<t\leq T$, and strongly continuous at $t=0$. 

\medskip

\item  For the range of parameters  $0<\gamma<3\leq p<+\infty$, we have
\begin{equation}\label{Relationship}
\vu \in M^{p,\gamma}_{t,x}\big([0,T]\times \Rt\big), \qquad \text{with} \quad 	\frac{\gamma}{p}-\frac{3}{p}+\frac{2}{3}< 0, \quad \text{and} \quad 0<T<+\infty. 
\end{equation}	
\end{enumerate} 

Then $(\vu, P)$ is a Leray solution in the sense of Definition \ref{Def-Leray-Solution}.
\end{Theorem}

Note that the sufficient condition (\ref{Relationship}) is imposed on the time interval $[0,T]$ for a \emph{fixed} time $T$. On the other hand, by Definition \ref{Def-Leray-Solution}, we know that Leray solutions are global-in-time and that their properties hold on $[0,T]$ for \emph{every} $T>0$.

\medskip

This is not merely a technical point, since in the proof of Theorem \ref{Th1} we show that $\vu$ satisfies the statements in Definition \ref{Def-Leray-Solution} on $[0,T]$ for the fixed time $T$ appearing in (\ref{Relationship}). Thereafter, by taking $\vu(T, \cdot)\in L^2(\Rt)$ as new initial data, the solution $\vu$ extends globally in time by following the classical Leray program (see \cite[Theorem $12.2$]{PLe}).

\medskip

{\bf Strategy of the proof and some comments}: Following  some of the ideas in \cite{Fernandez-Jarrin} and \cite{Fernandez-Lemarie},  the proof is based on sharp local energy estimates. We construct a well-prepared test function  (defined in expression (\ref{Test-Function}) below), which, in the spatial variable, is localized on the ball $|x|<R$. This function and the positivity  of the measure $\mu$ defined in (\ref{Local-Energy}), essentially lead us  to write 
\begin{equation*}
\begin{split}
&\,  \int_{|x|<\frac{R}{2}} | \vu(t,\cdot)|^2\, dx + 2 \int_{0}^{t}\int_{|x|<\frac{R}{2}} |\vn \otimes \vu(s,\cdot)|^2\, dx ds\\ \lesssim &\,   \| \vu_0\|^2_{L^2}+ T^{\frac{p-2}{p}}\,\, R^{-\frac{1}{3}} \, \| \vu\|^2_{M^{p,\gamma}_{t,x}}\, T^{\frac{p-3}{p}} R^{3\left(\frac{\gamma}{p}-\frac{3}{p}+\frac{2}{3}\right)} \| \vu \|^3_{M^{p,\gamma}_{t,x}}. 
 \end{split}
\end{equation*}
Here, the first point in Theorem \ref{Th1} is essentially technical to ensure that this estimate holds for any time $0<t\leq T$. On the other hand, we observe that the second point is the key assumption to control the right-hand side. Specifically, the global energy estimate (\ref{Energy-Inequality}) is obtained by letting $R\to +\infty$. Therefore, the condition on the parameters $\gamma$ and $p$ given in (\ref{Relationship}) appears naturally.  

\medskip

Additionally,  the constraint $0<\gamma<3\leq p<+\infty$ is due to technical requirements. On the one hand, setting $0<\gamma<3$ ensures   desired properties of Local Morrey spaces. See Lemma \ref{Lem-Tech-2} below. On the other hand, we also need that $3\leq p <+\infty$ to justify some estimates. 

\medskip

When $0<\gamma<3\leq p<+\infty$ is combined with the relationship given in (\ref{Relationship}), the parameter $p$ is ultimately restricted to the range $3 \leq p < \frac{9}{2}$. Here, the lower bound $p=3$ is of particular interest. In fact, from Definition \ref{Def-Weak-Suitable-Sol} we know that the velocity $\vu$ is an $L^3_{\rm loc}$-function. Thus, assuming that $\vu \in M^{3,\gamma}_{t,x}\big([0,T]\times\Rt \big)$, the definition of this space (see (\ref{Local-Morrey-tx})) yields, for any $0< T_1<T_2\leq T$ and $R\geq 1$,
\[  \left( \int_{T_1}^{T_2}\int_{|x|<R} | \vu(t,x)|^3\, dx dt  \right)^{\frac{1}{3}} \lesssim R^{\frac{\gamma}{3}}.  \]
We thus observe that the condition $\vu \in  M^{3,\gamma}_{t,x}\big([0,T]\times\Rt \big)$, or more generally $\vu \in M^{p,\gamma}_{t,x}\big([0,T]\times\Rt \big)$ with $p\geq 3$, essentially provides control of these local quantities which, together with the key suitability assumption, is sufficient to ensure that $\vu$ becomes a Leray solution.

\medskip

{\bf Some applications}: As mentioned, the assumption given in (\ref{Relationship}) gives a sufficiently general framework for the velocity $\vu$ to become a Leray solution.  In the following lines we give some examples of functional spaces, characterized by the parameters $p,q$, which are contained in the local Morrey space  $M^{p,\gamma}_{t,x}\big([0,T]\times \Rt \big)$. 

\medskip

The embeddings below  of course hold for more general ranges $p$ and $q$.  Nevertheless, we are only interesting in embeddings in this  space when  $p$ and $\gamma$  satisfy the relationship given in (\ref{Relationship}).  This fact leads to some restrictions on $p,q$ as detailed below. 

\medskip

Additionally,  for the reader's convenience, we recall the definition of  some these  functional spaces characterized by $p,q$. Additionally, the embeddings presented below  are proven in Appendix \ref{AppendixA}.  

\medskip

\begin{itemize} 
\item \emph{Lebesgue, Lorentz and Morrey spaces with $p\neq q$}. For a measurable function $\phi: \Rt \to \R$ and a parameter  $\lambda\geq0$, define the distribution function $d_\phi(\lambda)$ and the re-arrangement function $\phi^{*}$ by
\[  d_\phi(\lambda):= dx\big(\left\{ x \in \Rt: \, |\phi(x)|\geq \lambda \right\} \big) \quad \text{and} \quad \phi^{*}(z):=\inf\left\{ \lambda\geq 0:\, d_\phi(\lambda) \leq z \right\},  \]
respectively, where $dx$ denotes the Lebesgue measure. Then, for $1\leq q < +\infty$ and $1\leq r \leq +\infty$ the Lorentz space $L^{q,r}(\Rt)$ is the defined by the \emph{norm}\footnote{The quantity $\| \phi \|_{L^{q,r}}$ is often used as a norm. Nevertheless, it does not satisfy the triangle inequality.} 
\begin{equation*}
\| \phi \|_{L^{q,r}}:= \begin{cases}\vspace{2mm}
\ds{\frac{r}{q} \left( \int_{0}^{+\infty} \big( z^{\frac{1}{q}} \phi^{*}(z) \big)^{r} \, dz \right)^{\frac{1}{r}}}, & r<+\infty, \\
\ds{\sup_{z>0} \left( z^{\frac{1}{q}}\, \phi^{*}(z)\right)}, & r=+\infty. 
\end{cases}
\end{equation*} 

Then, for $3\leq p<q<\frac{9}{2}$ and $q<r\leq +\infty$ we have  
	\begin{equation}\label{Embedding1}
	L^p\big( [0,T], L^q(\Rt) \big) \subset L^p\big( [0,T], L^{q,r}(\Rt)\big)  \subset M^{p,\gamma}_{t,x}\big( [0,T]\times \Rt \big).  
	\end{equation}
	
On the other hand, 	 for  $1<p<q<+\infty$, the homogeneous Morrey space $\dot{M}^{p,q}(\Rt)$ is defined as the set of $L^p_{\rm loc}$-functions with the finite norm:
\begin{equation}\label{Homogeneous-Morrey}
\| \phi \|_{\dot{M}^{p,q}}:= \sup_{x_0 \in \Rt,\, R>0} R^{\frac{3}{q}} \,\left(  \frac{1}{R^3} \int_{|x-x_0|<R} |\phi(x)|^p\, dx \right)^{\frac{1}{p}}.
\end{equation}

Then,   for $3\leq p<q<\frac{9}{2}$, we have
\begin{equation}\label{Embedding2}
L^p\big( [0,T], \dot{M}^{p,q}(\Rt)\big) \subset M^{p,\gamma}_{t,x}\big( [0,T]\times \Rt\big).
\end{equation}	

Comparing with the previous related works, which essentially consider the functional $L^p_t L^q_x$,   the main interest in  embeddings (\ref{Embedding1}) and (\ref{Embedding2}) is  the fact that the $L^q_x$-space can be now substituted by a Lorentz spaces and a Morrey space. 

\medskip

\item \emph{Parabolic Morrey spaces}.  For  $1<p<q<+\infty$, the parabolic Morrey space $\mathcal{M}^{p,q}_{t,x}\big( \R \times \Rt\big)$ is defined as the set of   $(L^p_{t,x})_{\rm loc}$-functions with finite norm:
\begin{equation}\label{Parabollic-Morrey}
\| \phi \|_{\mathcal{M}^{p,q}_{t,x}}:= \sup_{x_0 \in \Rt, t_0 \in \R, R>0} \left( \frac{1}{R^{5\left( 1- \frac{p}{q}\right)}} \int_{|t-t_0|<R^2}\,\int_{|x-x_0|<R} |\phi(t,x)|^p\, dxdt \right)^{\frac{1}{p}}. 
\end{equation}

These spaces appear as a useful framework when studying qualitatively properties (related to the partial regularity theory) of solutions to the  Navier-Stokes equations \cite{Kukavica} and related models  \cite{Chamorro-Je,Chamorro-Mindrila,Chamorro-Llerena}. 

\medskip

In the following technical lemma, we show that  one can also consider these spaces as a sufficient condition  within  the setting of Theorem \ref{Th1}. 

\begin{Lemma}\label{Lem-Tech-Parabolic-Morrey} Let $0<T\leq 2$.  For $3\leq p<q<\frac{9}{2}$ assume that  $\mathds{1}_{[0,T]}\vu \in \mathcal{M}^{p,q}_{t,x}\big( \R \times \Rt\big)$. Then we have that $\vu \in M^{p,\gamma}_{t,x}\big( [0,T]\times \Rt \big)$ with $p$ and $\gamma$ satisfying (\ref{Relationship}). 
\end{Lemma}	
Here, the constraint $T\leq 2$ is essentially technical, as will be explained later in the proof of this lemma, which is given in Appendix \ref{AppendixA}. But, as explained below Theorem \ref{Th1}, these sufficient conditions for very weak suitable solutions to become Leray solutions are only required on a (possibly small) time interval.

\medskip

\item \emph{Lebesgue spaces and weighted Lebesgue spaces with $p=q$}.  For $\gamma> 0$, we define the weight function
\begin{equation}\label{Weighted-Lebesgue}
w_\gamma(x):= \frac{1}{(1+|x|)^\gamma},
\end{equation}
and for $1\leq p \leq +\infty$, we define the  weighted Lebesgue space  $L^p_{\gamma}(\Rt):= L^p(w_\gamma\, dx)$. 

\medskip

Then, for $0<\gamma<3\leq p<+\infty$ always satisfying (\ref{Relationship}),  we have
\begin{equation*}
L^p\big( [0,T], L^p(\Rt) \big) \subset L^p\Big( [0,T], L^p_\gamma (\Rt) \Big)\subset  M^{p,\gamma}_{t,x}\big( [0,T]\times \Rt \big). 
\end{equation*}
\end{itemize}	
For a reference of the proof of the last embedding see Lemma \ref{Lem-Tech-1} below. 

\medskip

Also returning to the previous works, when consider the  space $L^p_t L^p_x$, the main interest of this embedding follows from the fact that the $L^p_x$-space is now substituted by the more general $(L^p_\gamma)_x$-space. In this context, the value $p=4$ is of particular interest in connection with \cite[Lemma $2.1$]{Ding} and \cite[Theorem $1.1$]{Galdi}. 

\medskip

{\bf Characterization of the pressure term}: Observe that the local energy inequality given in expression (\ref{Local-Energy}) considers the pressure term $P$. To control this term, it is well-known that from the  divergence-free property of the velocity $\vu=(u_1, u_2, u_3)$,  we \emph{formally} have that 
\begin{equation}\label{Pressure}
\vn P = \vn \left( \sum_{i,j=1}^3 \mathcal{R}_i \mathcal{R}_j (u_i u_j) \right),
\end{equation}
where $\mathcal{R}_i:= (-\Delta)^{-\frac{1}{2}}\partial_i$ is the Riesz transform. 

\medskip

This characterization of the pressure term is well-understood in the framework of Lebesgue, Lorentz and homogeneous Morrey spaces, see \emph{e.g.} \cite[Lemma $4.2$]{Kato}. Nevertheless, the framework of Theorem \ref{Th1} considers the more general setting of the local Morrey spaces. 

\medskip

In our following (technical) result, under minimal assumptions, we show the validity of this characterization of the pressure term  within this general setting.  This result can also be of independent interest when looking for general frameworks in which the pressure $P$ is related to the velocity $\vu$ by (\ref{Pressure}) and similar expressions. See  for instance \cite{Fernandez-Lemarie-Pressure}. 

\begin{Proposition}\label{Prop:Pressure}  Let $\big(\vu,  P\big)$ be a very weak solution of the  Navier-Stokes equation  (\ref{NS}) in the sense of Definition \ref{Def-Very-Weak-Sol}.
	
\medskip

For the range of parameters $0<\gamma<\frac{3}{2}$ and $2<p<+\infty$, assume that 
\begin{equation}
\vu\in M^{p,\gamma}_{t,x}\big([0,T]\times \Rt \big).
\end{equation}
Then, the pressure $P\in \mathcal{D}'\big( [0,T]\times \Rt\big)$ is necessary related to the velocity $\vu$ by the expression (\ref{Pressure}).
\end{Proposition}

In contrast to Theorem \ref{Th1}, observe that this result uses the more general notion of very weak solutions since here the suitability property (given in Definition \ref{Def-Weak-Suitable-Sol}) is not longer required. 

\medskip

{\bf Corollaries}: It is well-known that uniqueness and regularity issues of Leray solution to the Navier-Stokes equations (\ref{NS}) are challenging open problems in all generality. Within the framework of Leray solutions obtained from Theorem \ref{Th1}, we can give some results in these directions. 

\medskip

Concerning the uniqueness issue,  we have  the following Liouville-type result. 

\begin{Corollary}\label{Cor:Liouville} Assume the same hypothesis of  Theorem \ref{Th1}.  If $\vu_0=0$ then we have $\vu=0$ on $[0,T]\times \Rt$. 
\end{Corollary}	
We thus show that the unique Leray solution of equation (\ref{NS}) associated with $\vu_0=0$ is the trivial solution. This result is a  direct implication of the global energy control (\ref{Energy-Inequality}) when setting $\vu_0=0$. 

\medskip

It is also natural to ask about the uniqueness of Leray solutions, obtained from Theorem \ref{Th1}, associated with any initial datum $\vu_0 \in L^2(\Rt)$. The setting of the $M^{p,\gamma}_{t,x}$-space used in this theorem alone does not seem to be sufficient to answer this question. Therefore, we need an additional assumption, also formulated within the general framework of local Morrey spaces.

\begin{Corollary}\label{Cor:Uniqueness} Under the same hypothesis of Theorem \ref{Th1}, for $0<T_1 \leq T$ assume in addition  that
\begin{equation}\label{Assumption-Uniqueness-Regularity}
\vu \in M^{q,\delta}_{t,x}\big([0,T_1] \times \Rt \big), \qquad \text{with} \quad  q \geq \frac{3p}{p-2} \quad  \text{and} \quad  \delta >0. 
\end{equation}
Then  $\vu$  is the unique  Leray solution corresponding to $\vu_0$ on $[0,T_1]\times \Rt$.  
\end{Corollary}	 

Concerning the regularity issue of Leray solutions obtained from Theorem \ref{Th1}, it can also be improved under the additional assumption used above.

\begin{Corollary}\label{Cor:Regularity} Under the same hypothesis of Theorem \ref{Th1}, assume  (\ref{Assumption-Uniqueness-Regularity}).  Then we have $\vu \in \mathcal{C}^{\infty}\big( ]0,T_1] \times \Rt \big)$.  
\end{Corollary}	

We observe that uniqueness and regularity hold on the (possibly smaller) time interval $[0,T_1]$ when assumption (\ref{Assumption-Uniqueness-Regularity}) is imposed. In this assumption, in contrast to the relationship (\ref{Relationship}), the parameter $q$ is not necessarily related to the parameter $\delta$.

\section{Preliminaries on the local Morrey spaces} 
We summarize here some known facts on the local Morrey spaces defined by expression  (\ref{Local-Morrey-tx}). We begin by noting that these spaces are closely related to the  weighted Lebesgue spaces, which were defined above by the weight function (\ref{Weighted-Lebesgue}).

\begin{Lemma}[Lemma $2.1$ of \cite{Fernandez-Jarrin}]\label{Lem-Tech-1}  Let $0< \gamma_1<\gamma_2$ and $1<p<+\infty$. the following continuous embeddings hold:
		\[
	L^p\left( [0,T], L^p_{\gamma_1}(\Rt)\right) \subset M^{p,\gamma_1}_{t,x}\big([0,T]\times \Rt \big) \subset L^p\left( [0,T], L^p_{\gamma_2}(\Rt)\right). 
	\]
\end{Lemma}	

Thereafter, recall the definition of the Riesz transform $\mathcal{R}_i:= (-\Delta)^{-\frac{1}{2}}\partial_i$ and the Hardy--Littlewood maximal function operator 
\begin{equation}\label{HL-Max-Op}
\mathcal{M}(\phi)(x):= \sup_{R>0} \frac{1}{|B(x,R)|}\int_{|x-y|<R} |\phi(y)| \, dy,
\end{equation}
where $| B(x,R) | \simeq R^3$ denotes the Lebesgue measure of the ball $B(x,R)$.

\begin{Lemma}\label{Lem-Tech-2}  Let $0<\gamma<3$ and  $1<p<+\infty$. There exists a numerical constant $C>0$, such that the following estimates hold.
		\begin{enumerate}
		\item $\ds{\|\mathcal{R}_i \phi \|_{L^p_t (L^p_\gamma)_x} \leq C \| \phi \|_{L^p_t(L^p_\gamma)_x}}$ and $\ds{\|\mathcal{M}(\phi)\|_{L^p_t (L^p_\gamma)_x}\leq C \| \phi \|_{L^p_t (L^p_\gamma)_x}}$. 
		
		\medskip
		
		\item $\ds{\|\mathcal{R}_i \phi \|_{M^{p,\gamma}_{t,x}} \leq C \| \phi\|_{M^{p,\gamma}_{t,x}}}$ and $\ds{\| \mathcal{M}(\phi)\|_{M^{p,\gamma}_{t,x}}\leq C \| \phi \|_{M^{p,\gamma}_{t,x}}}$.
	\end{enumerate}
\end{Lemma}	
\begin{proof} The first point was proven in \cite[Lemma 2]{Fernandez-Lemarie}. The second point was essentially proven in  \cite[Proposition $7.1$]{Fernandez-Lemarie-Pressure}. For the reader's convenience, we give a brief sketch  below. 
	
	\medskip
	
The key idea is that fact the local Morrey space $M^{p,\gamma}_{t,x}\big( [0,T]\times \Rt \big)$ can be obtained by real interpolation between the Lebesgue $L^p \big( [0,T], L^p(\Rt)\big)$ and the weighted Lebesgue space $L^p\big( [0,T], L^p_\delta(\Rt)\big)$. Specifically, essentially follows the same lines in the proof of \cite[Proposition $7.1$]{Fernandez-Lemarie-Pressure}, one can prove that 
\[ M^{p,\gamma}_{t,x} = \big[  L^p_t L^p_x, \, L^p_t (L^p_\delta)_x \big]_{\frac{\gamma}{\delta}, \infty}, \quad \text{with} \quad 0<\gamma<\delta,\]
where the norms $\| \cdot \|_{M^{p,\gamma}_{t,x}}$ and $\| \cdot \|_{[L^p_t L^p_x, \,  L^p_t (L^p_\delta)_x]}$ are equivalents.  

\medskip

Recall that the interpolation space $\big[  L^p_t L^p_x, \, L^p_t (L^p_\delta)_x \big]_{\frac{\gamma}{\delta}, \infty}$ is defined as the set of measurable functions $\phi(t,x)$ such that the series  $\ds{\phi = \sum_{k \in \mathbb{Z}} \phi_k}$ convergences in the strong topology of $L^p_t L^p_x + L^p_t (L^p_\delta)_x$. Here, $\phi_k \in L^p_t L^p_x \cap L^p_t (L^p_\delta)_x$ and, for $\gamma<\delta$, we  have $\ds{\sup_{k \in \mathbb{Z}}} \left( 2^{-k\frac{\gamma}{\delta}}\, \max\left( \| \phi_k \|_{L^p_t L^p_x}, \, 2^k \| \phi_k \|_{L^p_t (L^p_\delta)_x} \right) \right) <+\infty$.  Moreover, we define 
\[   \| \phi \|_{[L^p_t L^p_x,\,  L^p_t (L^p_\delta)_x]} := \min_{(\phi_k)_{k\in \mathbb{Z}} \subset L^p_t L^p_x + L^p_t (L^p_\delta)_x}\, \left( \sup_{k\in \mathbb{Z}} \left( 2^{-k \frac{\gamma}{\delta}}\| \phi_k \|_{L^p_t L^p_x}+2^{k \left(1-\frac{\gamma}{\delta} \right)} \| \phi_k \|_{L^p_t (L^p_\delta)_x} \right) \right). \]

Thus, the second point of Lemma \ref{Lem-Tech-2} follows from the equivalence $\| \cdot \|_{M^{p,\gamma}_{t,x}} \simeq \| \cdot \|_{[L^p_t L^p_x,\, L^p_t (L^p_\delta)_x]}$ and the the continuity of the operators $\mathcal{R}_i$ and $\mathcal{M}(\cdot)$ in the spaces $L^p_t L^p_x$ and $L^p_t (L^p_\delta)_x$. In particular, the continuity property in this last space  is ensured by the first point of Lemma \ref{Lem-Tech-2}.  \end{proof}	

\section{Proof of Proposition \ref{Prop:Pressure}}
From the velocity $\vu=(u_1, u_2, u_3)$  we define the expression
\begin{equation}\label{Q}
Q:= \sum_{i,j=1}^{3}\mathcal{R}_i \mathcal{R}_j (u_i u_j).
\end{equation}
Therefore, from the given pressure term $P \in \mathcal{D}'\big( [0,T]\times \Rt \big)$ in the equation (\ref{NS}),  we will prove that
\begin{equation*}
\vn (P-Q)=0.
\end{equation*}

To this end, we divide the proof into the following technical lemmas.

\begin{Lemma}\label{Lem-Tech-Q}
	Assume that $\vu \in M^{p,\gamma}_{t,x}\big( [0,T]\times \Rt \big)$ with $0<\gamma<3$ and $2<p<+\infty$. Then it holds that $Q \in L^{\frac{p}{2}}\left( [0,T], L^{\frac{p}{2}}_{\delta}(\Rt)\right)$ for some $\gamma<\delta <3$.
\end{Lemma}

\begin{proof} 
	From the first part of Lemma \ref{Lem-Tech-1} (with $\gamma_1=\gamma$ and $\gamma_2=\delta$), we have  $\vu\in L^{p}\big( [0,T], L^p_\delta(\Rt)\big)$. Then, by applying H\"older's inequality, it follows that $\vu \otimes \vu \in L^{\frac{p}{2}}\big([0,T], L^{\frac{p}{2}}_\delta (\Rt)\big)$. Finally, by the first point of Lemma  \ref{Lem-Tech-2}, we obtain that $Q\in L^{\frac{p}{2}}\left( [0,T], L^{\frac{p}{2}}_{\delta}(\Rt)\right)$.
\end{proof}	

On the other hand, we introduce the following test functions. Let $\alpha(\cdot)\in \mathcal{C}^\infty_0(\R)$ be a positive function such that $\alpha(t)=0$ when $|t|>1$ and $\int_{\R}\alpha(t)\,dt=1$. Similarly, let $\varphi\in \mathcal{C}^{\infty}_{0}(\Rt)$ be a positive function such that $\int_{\Rt}\varphi(x)\,dx=1$. Then, for any $\varepsilon>0$ sufficiently small, we consider the family of approximations of the identity
\[
\frac{1}{\varepsilon^4} \,\alpha\left( \frac{t}{\varepsilon} \right)\varphi \left( \frac{x}{\varepsilon}\right).
\]

\medskip

From the terms $P$ and $Q$, note that we have
\[
\left( \frac{1}{\varepsilon^4}\, \alpha\left( \frac{\cdot}{\varepsilon} \right)\varphi \left( \frac{\cdot}{\varepsilon}\right) \right)\ast \vn P \in \mathcal{D}'\big(]\varepsilon, T-\varepsilon[\times \Rt \big), \quad
\left( \frac{1}{\varepsilon^4}\, \alpha\left( \frac{\cdot}{\varepsilon} \right)\varphi \left( \frac{\cdot}{\varepsilon}\right) \right)\ast \vn Q \in \mathcal{D}'\big(]\varepsilon, T-\varepsilon[\times \Rt \big).
\]
Consequently, for any fixed $\varepsilon < t < T-\varepsilon$, it holds that
\[
\left( \frac{1}{\varepsilon^4} \alpha\left( \frac{\cdot}{\varepsilon} \right)\varphi \left( \frac{\cdot}{\varepsilon}\right) \right)\ast \big( \vn P- \vn Q \big)(t,\cdot):= \big( \vn P- \vn Q \big)_\varepsilon(t,\cdot) \in \mathcal{D}'(\Rt).
\]
In this context, we need to verify that this expression also belongs to the space $L^p_\delta(\Rt)+L^{\frac{p}{2}}_\delta (\Rt)$, where the parameter $\delta$ is given in Lemma \ref{Lem-Tech-Q}.

\begin{Lemma}\label{Lem-Tech-Q-2}
	Let $\vu \in M^{p,\gamma}_{t,x}\big( [0,T]\times \Rt\big)$ (with $0<\gamma<3$ and $2<p<+\infty$), $P \in \mathcal{D}'\big( [0,T]\times \Rt \big)$ given by the  equation (\ref{NS}), and $Q \in L^{\frac{p}{2}}\left( [0,T], L^{\frac{p}{2}}_{\delta}(\Rt)\right)$ from Lemma \ref{Lem-Tech-Q}. Then it holds that $ \big( \vn P - \vn Q \big)_\varepsilon(t,\cdot) \in L^p_\delta (\Rt)+  L^{\frac{p}{2}}_\delta (\Rt)$.
\end{Lemma}

\begin{proof}
	For simplicity, we write $\big( \vn P - \vn Q \big)_\varepsilon(t,\cdot)$ as $(\alpha \varphi)\ast \big( \vn P  - \vn Q \big)(t,\cdot)$. From equation in (\ref{NS}), we write
	\begin{equation}\label{Eq-nabla-P}
	\vn P = - \partial_t \vu + \Delta \vu - \text{div}(\vu \otimes \vu).
	\end{equation}
	Hence,
	\begin{equation*}
	(\alpha \varphi)\ast \big( \vn P - \vn Q  \big)(t,\cdot)  =  \underbrace{(\alpha \varphi)\ast\Big( -\partial_t + \Delta \Big)\vu(t,\cdot)}_{I_1(t,\cdot)} +  \underbrace{(\alpha \varphi)\ast\Big(-\text{div}(\vu\otimes \vu)\Big)(t,\cdot)}_{I_2(t,\cdot)} - \underbrace{ (\alpha \varphi)\ast \vn Q (t,\cdot)}_{I_3(t,\cdot)},
	\end{equation*}
	where we must verify that each term on the right-hand side belongs to the space $L^{\frac{p}{2}}_\delta(\Rt)$.
	
	\medskip
	
	For $I_1(t,\cdot)$, we write
	\[
	I_1(t,\cdot)= \Big( (-\partial_t \alpha) \varphi + \alpha \Delta \varphi \Big)\ast \vu (t,\cdot).
	\]
	Recall that for $\phi \in \mathcal{C}^{\infty}_0(\Rt)$ and a given function $f$, we have the pointwise control
	\[
	|(\phi \ast f)(x)|\leq C_\phi \mathcal{M}(f)(x),
	\]
	where $C_\phi>0$ is a constant depending on $\phi$, and $\mathcal{M}(f)$ denotes the Hardy--Littlewood maximal operator defined in (\ref{HL-Max-Op}). Therefore, since $\vu \in M^{p,\gamma}_{t,x}\big([0,T]\times \Rt \big)\subset  L^p\big([0,T], L^p_\delta(\Rt)\big)$ (from the first part of Lemma \ref{Lem-Tech-1}), applying the first part of Lemma \ref{Lem-Tech-2}, we obtain that $I_1(t,\cdot) \in  L^p_\delta(\Rt)$.
	
	\medskip
	
	For $I_2(t,\cdot)$, we write
	\[
	I_2(t,\cdot)= (\alpha \vn \varphi) \ast \big(-\vu \otimes \vu\big),
	\]
	and since $\vu \otimes \vu \in L^{\frac{p}{2}}\big( [0,T], L^{\frac{p}{2}}_{\delta}(\Rt)\big)$, by the same arguments as above, we obtain that $I_2(t,\cdot)\in L^{\frac{p}{2}}_{\delta}(\Rt)$.
	
	\medskip
	
	The term $I_3(t,\cdot)$  also follows the same ideas, yielding that  $I_3(t,\cdot)\in L^{\frac{p}{2}}_\delta(\Rt)$. 
\end{proof}	 

{\bf End of the proof of Proposition \ref{Prop:Pressure}}. By Lemma \ref{Lem-Tech-Q-2}, we have $\big( \vn P - \vn Q \big)_\varepsilon(t,\cdot) \in L^p_\delta (\Rt)+  L^{\frac{p}{2}}_\delta (\Rt)$ and therefore $\big( \vn P - \vn Q \big)_\varepsilon(t,\cdot) \in \mathcal{S}'(\Rt)$.

\medskip

On the other hand, from the expression (\ref{Q}) we have
\[
-\Delta Q= \text{div}\,\text{div}\big(\vu \otimes \vu\big).
\]
Moreover, from equation (\ref{Eq-nabla-P}), we write
\[
\vn P + \text{div}(\vu \otimes \vu)= - \partial_t \vu + \Delta \vu,
\]
and since $\text{div}(\vu)=0$, applying the divergence operator, we find that $\Delta(P-Q)=0$, hence
\[
\Delta  \big( \vn P - \vn Q \big)_\varepsilon(t,\cdot)=0. 
\]

Consequently, the expression $ \big( \vn P - \vn Q \big)_\varepsilon(t,\cdot)$ is a polynomial. Since it also belongs to the space $L^p_\delta (\Rt)+  L^{\frac{p}{2}}_\delta (\Rt)$, we conclude that $\big( \vn P - \vn Q \big)_\varepsilon (t,\cdot)=0$. Finally, using the approximation of the identity introduced above, we write
\[
\vn (P-Q)(t,\cdot)= \lim_{\varepsilon\to 0^{+}} \big( \vn P - \vn Q \big)_\varepsilon (t,\cdot)=0.
\]
Proposition \ref{Prop:Pressure} is now proven.

\section{Proof of Theorem \ref{Th1}} 
As mentioned, the main strategy of the proof is to construct a convenient test function $\Phi \in \mathcal{C}^{\infty}_0 \big( [0,T]\times \Rt\big)$ to be applied to the local energy equality (\ref{Local-Energy}).

\medskip

On the one hand, let $\alpha(\cdot)\in \mathcal{C}^{\infty}(\R)$ be a function such that 
\begin{equation*}
\alpha(t)= \begin{cases}
0, & \text{when} \quad -\infty <t < \frac{1}{2}, \\
1, & \text{when} \quad 1 <t < +\infty. 
\end{cases}
\end{equation*}
For $0< t_0<t_1\leq T$ and for a sufficiently small parameter $\varepsilon>0$ (which we will precise later), define the function 
\begin{equation*}
\alpha_{t_0,t_1,\varepsilon}(t):= \alpha\left( \frac{t-t_0}{\varepsilon} \right)- \alpha\left( \frac{t-t_1}{\varepsilon} \right).
\end{equation*}
From the definition of the function $\alpha(t)$ given above, we have the  that  $\ds{\lim_{\varepsilon\to 0^{+}} \alpha_{t_0,t_1,\varepsilon}(t)=\mathds{1}_{[t_0, t_1]}(t)}$ for \emph{a.e.} $t_0\leq t \leq t_1$.  Additionally, always by the definition of $\alpha(t)$, it follows that 
\[ \frac{d}{dt}\alpha_{t_0,t_1,\varepsilon}(t) =  \frac{1}{\varepsilon}\alpha'\left( \frac{t-t_0}{\varepsilon} \right)-\frac{1}{\varepsilon} \alpha'\left( \frac{t-t_1}{\varepsilon} \right),\]
where each term on the right-hand side defines a  family of approximation of the identity at $t_0$ and at $t_1$, respectively.  Finally, we have that $\text{supp}\left( \alpha_{t_0,t_1,\varepsilon}(\cdot) \right)\subseteq \left[ t_0-\frac{\varepsilon}{2}, t_1+\varepsilon \right]$. Consequently,  we assume that  $0<\varepsilon<\min \left( \frac{t_0}{2}, T-t_1 \right)$ yielding that  $\left[ t_0-\frac{\varepsilon}{2}, t_1+\varepsilon \right] \subset [0,T]$. 

\medskip

On the other hand, let $\varphi \in \mathcal{C}^{\infty}_0(\Rt)$ be a positive and radial function such that 
\begin{equation*}
\varphi(x)= \begin{cases}
1, & \text{when}  \quad  |x|< \frac{1}{2}, \\
0, & \text{when}  \quad  |x| \geq 1. 
\end{cases}
\end{equation*}
For a parameter $R\geq 1$, define the function 
\begin{equation*}
\varphi_R(x):= \varphi \left( \frac{x}{R} \right).
\end{equation*}
From the definition of $\varphi(x)$ it follows that $\varphi_R(x)=1$ when $|x|<\frac{R}{2}$, $\varphi_R(x)=0$ when $|x|\geq R$. Additionally, for some constant $C>0$ independent of $R$, we have that  $\ds{\left\| \vn \varphi_R \right\|_{L^\infty} \leq \frac{C}{R}}$ and  $\ds{\| \Delta \varphi_R \|_{L^\infty} \leq \frac{C}{R^2}}$. 

\medskip

With these functions, we then define the test function 
\begin{equation}\label{Test-Function}
\Phi(t,x):= \alpha_{t_0,t_1,\varepsilon}(t)\, \varphi_R(x). 
\end{equation}

Applying this test function into the local energy estimate (\ref{Local-Energy}), recalling that the distribution $\mu$ is a non-negative locally finite measure,  after  rearranging terms  we find that 
\begin{equation*}
\begin{split}
&\, \int_{0}^{T}\int_{\Rt} \partial_t  |\vu|^2\,\alpha_{t_0,t_1,\varepsilon}\varphi_R\, dxdt  + 2\int_{0}^{T}\int_{\Rt}  |\vn \otimes \vu |^2 \,\alpha_{t_0,t_1,\varepsilon}\varphi_R\, dxdt \\
\leq &\, \int_{0}^{T}\int_{\Rt} \Delta |\vu|^2\, \alpha_{t_0,t_1,\varepsilon}\varphi_R\, dxdt - \int_{0}^{T}\int_{\Rt} \text{div}\Big( |\vu|^2 \vu  + 2 P \vu\Big)\, \alpha_{t_0,t_1,\varepsilon}\varphi_R\, dxdt. 
\end{split}
\end{equation*}
Integrating by parts the first, third and fourth term, we obtain 
\begin{equation*}
\begin{split}
&\,- \int_{0}^{T}\int_{\Rt} |\vu|^2\, \left(\frac{d}{dt}\alpha_{t_0,t_1,\varepsilon}\right)\varphi_R\, dxdt  + 2\int_{0}^{T}\int_{\Rt}  |\vn \otimes \vu |^2\,\alpha_{t_0,t_1,\varepsilon}\varphi_R\, dxdt \\
\leq &\, \int_{0}^{T}\int_{\Rt} |\vu|^2\, \alpha_{t_0,t_1,\varepsilon}\,\Delta\varphi_R\, dxdt + \int_{0}^{T}\int_{\Rt} \Big( |\vu|^2 \vu  + 2 P \vu\Big)\, \alpha_{t_0,t_1,\varepsilon} \vn\varphi_R\, dxdt. 
\end{split}
\end{equation*}
Finally, taking the limit as $\varepsilon\to 0^{+}$, from the Lebesgue dominated convergence theorem and the fact that $\ds{\lim_{\varepsilon\to 0^{+}} \alpha_{t_0,t_1,\varepsilon}(t)=\mathds{1}_{[t_0, t_1]}(t)}$, for \emph{a.e.} $t_0\leq t \leq t_1$, we write
\begin{equation*}
\begin{split}
&\,- \lim_{\varepsilon\to 0^{+}} \int_{0}^{T}\int_{\Rt}  |\vu|^2\, \left(\frac{d}{dt}\alpha_{t_0,t_1,\varepsilon}\right)\varphi_R\, dxdt  + 2\int_{t_0}^{t_1}\int_{\Rt} |\vn \otimes \vu |^2 \,\varphi_R\, dxdt \\
\leq &\, \int_{t_0}^{t_1}\int_{\Rt} |\vu|^2\, \Delta\varphi_R\, dxdt + \int_{t_0}^{t_1}\int_{\Rt} \Big( |\vu|^2 \vu  + 2 P \vu\Big) \,  \vn\varphi_R\, dxdt. 
\end{split}
\end{equation*}

We study the first term on the left-hand side separately. To simplify our writing, denote 
\[ \Lambda_R(t):= \int_{\Rt} |\vu(t,x)|^2\, \varphi_R\, dx,  \]
hence we write
\[ \int_{0}^{T}\int_{\Rt}  |\vu(t,x)|^2\, \left(\frac{d}{dt}\alpha_{t_0,t_1,\varepsilon}(t)\right)\varphi_R(x)\, dxdt = \int_{0}^{T} \left(\frac{d}{dt}\alpha_{t_0,t_1,\varepsilon}(t)\right)\, \Lambda_R(t) dt.  \]
Recalling that $\ds{\frac{d}{dt}\alpha_{t_0,t_1,\varepsilon}(t)}$ is the difference between the family of approximation of identity at $t_0$ and at $t_1$, assuming that $t_0$ and $t_1$ are Lebesgue points of the function $\Lambda_R(t)$, we obtain 
\[ \lim_{\varepsilon\to 0^{+}} \int_{0}^{T} \left(\frac{d}{dt}\alpha_{t_0,t_1,\varepsilon}(t)\right)\, \Lambda_R(t) dt= \Lambda_R(t_0)-\Lambda_R(t_1). \]

Returning to the previous estimate, for \emph{a.e.} $t_0 \leq t \leq t_1$ we write
\begin{equation*}
\begin{split}
&\, \int_{\Rt} |\vu(t_1, \cdot)|^2\, \varphi_Rdx  + 2\int_{t_0}^{t_1}\int_{\Rt}  |\vn \otimes \vu(t,\cdot) |^2 \,\varphi_R\, dxdt \\
\leq &\int_{\Rt} |\vu(t_0, \cdot)|^2\, \varphi_Rdx + \int_{t_0}^{t_1}\int_{\Rt}  |\vu|^2\, \Delta\varphi_R\, dxdt + \int_{t_0}^{t_1}\int_{\Rt} \Big( |\vu|^2 \vu  + 2 P \vu\Big)\,  \vn\varphi_R\, dxdt. 
\end{split}
\end{equation*}

Now, we study more in detail the times $t_0$ and $t_1$. For the time $t_0$, recall that by the first assumption in Theorem \ref{Th1} we have that $\vu(t,\cdot)$  are strongly continuous at $t=0$. Therefore, we can let $t_0\to 0^{+}$. For the time $t_1$, always by this assumption we have that  $\vu(t,\cdot)$ is weakly continuous  in $(0,T]$.   Consequently, for any $0<t\leq T$  we have  $\ds{\Lambda_R(t) \leq \liminf_{t_1 \to t} \Lambda_R(t_1)}$,  and we can substitute $t_1$ by $t$. 

\medskip

With  these facts, for  $0<t\leq T$, we obtain
\begin{equation*}
\begin{split}
&\, \int_{\Rt} |\vu(t, \cdot)|^2\, \varphi_Rdx  + 2\int_{0}^{t}\int_{\Rt}  |\vn \otimes \vu(s,\cdot) |^2 \,\varphi_R\, dxds \\
\leq &\int_{\Rt} |\vu_0|^2\, \varphi_Rdx + \int_{0}^{t}\int_{\Rt}  |\vu|^2\, \Delta\varphi_R\, dxds + \int_{0}^{t}\int_{\Rt} \Big( |\vu|^2 \vu  + 2 P \vu\Big)\, \vec{\nabla}\varphi_R\, dxds. 
\end{split}
\end{equation*}

Finally, recalling the localization properties of the test function $\varphi_R$, where we denote $B_{R}:= \left\{ x \in \Rt: \, |x|<R \right\}$ and $C_R:= \left\{ x \in \Rt: \, \frac{R}{2}<|x|<R \right\}$, and as $\vu_0\in L^2(\Rt)$,  for any $R\geq 1$ we have
\begin{equation}\label{Estim-Base}
\begin{split}
&\, \int_{B_{\frac{R}{2}}} |\vu(t, \cdot)|^2 \,dx  + 2\int_{0}^{t}\int_{B_{\frac{R}{2}}}|\vn \otimes \vu(s,\cdot) |^2 \, dxds \\
\leq &\int_{B_R} |\vu_0|^2\,\varphi_Rdx + \int_{0}^{t}\int_{C_R} |\vu|^2\, \Delta\varphi_R\, dxds + \int_{0}^{t}\int_{C_R}  \Big( |\vu|^2 \vu  + 2 P \vu\Big) \,  \vn\varphi_R\, dxds\\
\leq & \| \vu_0 \|^2_{L^2} + \int_{0}^{t}\int_{C_R}  |\vu|^2\, \Delta\varphi_R\, dxds + \int_{0}^{t}\int_{C_R} \Big( |\vu|^2 \vu  + 2 P \vu\Big)\,  \vn\varphi_R\, dxds.
\end{split}
\end{equation}

In the forthcoming technical lemmas, we estimate each term on the right-hand side. To this end, we will use a generic constant $C>0$, which may change from one line to the next but it is independent of the parameter $R$.

\begin{Lemma}\label{Lem-Tech-Energy-1}  As $\vu \in M^{p,\gamma}_{t,x}\big( [0,T]\times \Rt \big)$, with $0<\gamma<3\leq p<+\infty$ satisfying (\ref{Relationship}), for any $0<t\leq T$  we have
	\[ \left| \int_{0}^{t}\int_{C_R}  |\vu|^2 \, \Delta\varphi_R\, dxds  \right| \leq C\, T^{\frac{p-2}{p}}\, R^{-\frac{1}{3}}\,  \| \vu \|^2_{M^{p,\gamma}_{t,x}}.  \]
\end{Lemma}
\begin{proof}  As $\| \Delta \varphi_R \|_{L^\infty} \leq \frac{C}{R^2}$,  $p\geq 3$, and applying H\"older inequalities in the spatial variable (with $1=\frac{2}{p}+ \frac{p-2}{p}$), for $0\leq t \leq T$ we have 
	\begin{equation*}
	\begin{split}
&\, 	\left| \int_{0}^{t}\int_{C_R}  |\vu|^2\, \Delta\varphi_R\, dxds  \right| \leq \,  \frac{C}{R^2}\, \int_{0}^{t}\int_{C_R}  |\vu|^2\, dxds\leq \, C\, R^{1-\frac{6}{p}} \int_{0}^{t} \left(  \int_{C_R}  |\vu|^p\, dx\right)^{\frac{2}{q}} ds \\
	\leq  & \,  C\, R^{1-\frac{6}{p}} \left( \int_{0}^{t}  \int_{C_R}  |\vu|^p\, dxds \right)^{\frac{2}{p}} t^{\frac{p-2}{p}}\leq \,  C\, R^{1-\frac{6}{p}} \left( \int_{0}^{t}  \int_{C_R}  |\vu|^p\, dxds \right)^{\frac{2}{p}} T^{\frac{p-2}{p}}.
	\end{split}
	\end{equation*}

Now, for  the parameter  $\gamma$, inside the last integral we multiply and divide by $R^\gamma$, yielding that 
\begin{equation*}
\begin{split}
&\, C\, R^{1-\frac{6}{p}} \left( \int_{0}^{t}  \int_{C_R}  |\vu|^p\, dxds \right)^{\frac{2}{p}} T^{\frac{p-2}{p}}\\
=&\, C\, R^{1-\frac{6}{p}+\frac{2\gamma}{p}}\, \left( \frac{1}{R^\gamma} \int_{0}^{t}  \int_{C_R}  |\vu|^p\, dxds \right)^{\frac{2}{p}} T^{\frac{p-2}{p}}\\
=&\, C\, R^{2\left(\frac{1}{2}-\frac{3}{p}+\frac{\gamma}{p}\right)}\, \left( \frac{1}{R^\gamma} \int_{0}^{t}  \int_{C_R}  |\vu|^p\, dxds \right)^{\frac{2}{p}} T^{\frac{p-2}{p}}\\
=&\, C\, R^{2\left(\frac{\gamma}{p}-\frac{3}{p}+\frac{2}{3}-\frac{1}{6}\right)}\, \left( \frac{1}{R^\gamma} \int_{0}^{t}  \int_{C_R} |\vu|^p\, dxds \right)^{\frac{2}{p}} T^{\frac{p-2}{p}}.
\end{split}
\end{equation*}

Recall that by (\ref{Relationship}) we have $\ds{\frac{\gamma}{p}-\frac{3}{p}+ \frac{2}{3}< 0}$.  Using this fact, and the definition of the norm $\| \cdot \|_{M^{p,\gamma}_{t,x}}$ given in (\ref{Local-Morrey-tx}), for any $R\geq 1$  we finally obtain
\begin{equation*} C\, R^{2\left(\frac{2}{3}-\frac{3}{p}+\frac{\gamma}{p}-\frac{1}{6}\right)}\, \left( \frac{1}{R^\gamma} \int_{0}^{t}  \int_{C_R}  |\vu|^p\, dxds \right)^{\frac{2}{p}} T^{\frac{p-2}{p}} \leq  C\, R^{-\frac{1}{3}}\, \| \vu \|^2_{M^{p,\gamma}_{t,x}}\, T^{\frac{p-2}{p}}. 
\end{equation*} 
\end{proof}	 	

\begin{Remark} Note that the control proven in Lemma \ref{Lem-Tech-Energy-1} actually holds under the more relaxed relationship between $p$ and $\gamma$: $\frac{\gamma}{p}-\frac{3}{p}+\frac{2}{3}< \frac{1}{6}$. Nevertheless, the more restrictive relationship (\ref{Relationship}) naturally appears in the control of the following expression.
\end{Remark}	

\begin{Lemma}\label{Lem-Tech-Energy-2} As $\vu\in M^{p,\gamma}_{t,x}\big( [0,T]\times \Rt \big)$, with $0<\gamma<3\leq p<+\infty$ satisfying (\ref{Relationship}),  for any $0<t\leq T$ we have
	\begin{equation*}
\left|  \int_{0}^{t}\int_{C_R}  \Big( |\vu|^2 \vu  + 2 P \vu\Big)\,  \vn\varphi_R\, dxds  \right| 
	\leq \, C\, T^{\frac{p-3}{p}}\,  R^{3\left( \frac{\gamma}{p}-\frac{3}{p}+\frac{2}{3}\right)}\, \| \vu \|^3_{M^{p,\gamma}_{t,x}}.
	\end{equation*}
\end{Lemma}	
\begin{proof} 
We split
\begin{equation*}
\left|  \int_{0}^{t}\int_{C_R}  \text{div}\Big( |\vu|^2 \vu  + 2 P \vu\Big)\,  \vn\varphi_R\, dxds  \right| \leq \,  \int_{0}^{t}\int_{C_R}  |\vu|^2\,|\vu|\, dx ds  + 2\int_{0}^{t}\int_{C_R} |P||\vu|\, | \vn \varphi_R |\, dx ds=:  I_1+I_2,
\end{equation*}	
where we must estimate each term on the right-hand side.  

\medskip

For $I_1$, using the estimate $\| \vn \varphi_R \|_{L^\infty}\leq \frac{C}{R}$ together with the H\"older inequalities in the spatial variable (with $1=\frac{2}{p}+ \frac{p-2}{p}$), we have 
\begin{equation*}
I_1 \leq \, \frac{C}{R} \, \int_{0}^{t} \left(  \int_{C_R} |\vu|^p \, dx \right)^{\frac{2}{p}}\, \left( \int_{C_R} |\vu|^{\frac{p}{p-2}}\, dx  \right)^{\frac{p-2}{p}}\, ds. 
\end{equation*}
For the second term above,  as $3\leq p<+\infty$ we have $\frac{p}{p-2}\leq 3 \leq p$. Then, we can write
\[ \left( \int_{C_R} |\vu|^{\frac{p}{p-2}}\, dx  \right)^{\frac{p-2}{p}} \leq C\, R^{3\left( \frac{p-2}{p} - \frac{1}{p}\right)}\, \left( \int_{C_R}|\vu|^p\, dx  \right)^{\frac{1}{p}} =C\, R^{3-\frac{9}{p}}\, \left( \int_{C_R}|\vu|^p\, dx  \right)^{\frac{1}{p}}.  \] 
Returning to the previous estimate, rearranging terms  we find that
\begin{equation*}
\begin{split}
I_1\leq  &\,C\, R^{2-\frac{9}{p}} \int_{0}^{t} \left( \int_{C_R} |\vu|^p\, dx \right)^{\frac{2}{p}}\, \left( \int_{C_R}|\vu|^p\, dx \right)^{\frac{1}{p}}\, ds\\
=&\,C\, R^{2-\frac{9}{p}+\frac{3\gamma}{p}} \int_{0}^{t} \left( \frac{1}{R^\gamma} \int_{C_R} |\vu|^p\, dx \right)^{\frac{2}{p}}\, \left(\frac{1}{R^\gamma} \int_{C_R}|\vu|^p\, dx \right)^{\frac{1}{p}}\, ds\\
=&\, C\,  R^{3\left(\frac{\gamma}{p}-\frac{3}{p}+\frac{2}{3}\right)} \int_{0}^{t} \left( \frac{1}{R^\gamma} \int_{C_R} |\vu|^p\, dx \right)^{\frac{2}{p}}\, \left(\frac{1}{R^\gamma} \int_{C_R}|\vu|^p\, dx \right)^{\frac{1}{p}}\, ds.
\end{split}
\end{equation*}
Applying the H\"older inequality in the time variable (with $1=\frac{2}{p}+\frac{1}{p}+\frac{p-3}{p}$) and the definition of $\|\cdot \|_{M^{p,\gamma}_{t,x}}$ given in (\ref{Local-Morrey-tx}),  for any $0\leq t \leq T$ we obtain 
\begin{equation}\label{Estim-I1}
\begin{split}
I_1 \leq &\, C\, R^{3\left(\frac{\gamma}{p}-\frac{3}{p}+\frac{2}{3}\right)}\left(\frac{1}{R^\gamma} \int_{0}^{t}\int_{C_R} |\vu|^p\, dxds \right)^{\frac{2}{p}}\,  \left( \frac{1}{R^\gamma}  \int_{0}^{t}\int_{C_R} |\vu|^p\, dxds \right)^{\frac{1}{p}} t^{\frac{p-3}{p}}\\
\leq &\, C\, R^{3\left(\frac{\gamma}{p}-\frac{3}{p}+\frac{2}{3}\right)}  \left(\frac{1}{R^\gamma} \int_{0}^{t}\int_{C_R} |\vu|^p\, dxds \right)^{\frac{2}{p}}  \left(\frac{1}{R^\gamma} \int_{0}^{t}\int_{C_R} |\vu|^p\, dxds \right)^{\frac{1}{p}}\, T^{\frac{p-3}{p}}\\
\leq &\, C\, R^{3\left(\frac{\gamma}{p}-\frac{3}{p}+\frac{2}{3}\right)}\, \| \vu\|^3_{M^{p,\gamma}_{t,x}}\, T^{\frac{p-3}{p}}. 
\end{split}
\end{equation}

For $I_2$,  first recall that by Proposition \ref{Prop:Pressure} we have the identity   $\vn P=\vn Q$, where the term $Q$ is  given in (\ref{Q}). Therefore, in equation  (\ref{NS}), we can write $\vn Q$ instead of $\vn P$ yielding after some standard computation  that 
\[ I_2:= 2 \int_{0}^{t}\int_{C_R} |Q||\vu||\vn \varphi_R|\, dxds. \]

Using the same arguments as the term $I_1$, we write 
\begin{equation*}
\begin{split}
 I_2 \leq  &\, \frac{C}{R}\, \int_{0}^{t} \left( \int_{C_R} |Q|^{\frac{p}{2}}\, dx \right)^{\frac{2}{p}}\, \left(  \int_{C_R}|\vu|^{\frac{p}{p-2}}\, dx \right)^{\frac{p-2}{p}}\, ds \\
 \leq &\, C\, R^{2-\frac{9}{p}}\, \int_{0}^{t} \left( \int_{C_R} |Q|^{\frac{p}{2}}\, dx \right)^{\frac{2}{p}}\, \left(  \int_{C_R}|\vu|^{p}\, dx \right)^{\frac{1}{p}}\, ds \\
 \leq &\, C\, R^{3\left( \frac{2}{3}-\frac{3}{p}+\frac{\gamma}{p}\right)}\, \int_{0}^{t} \left( \frac{1}{R^\gamma} \int_{C_R} |Q|^{\frac{p}{2}}\, dx \right)^{\frac{2}{p}}\, \left(  \frac{1}{R^{\gamma}} \int_{C_R}|\vu|^{p}\, dx \right)^{\frac{1}{p}}\, ds \\
 \leq &\, C\,  R^{3\left(\frac{\gamma}{p}-\frac{3}{p}+\frac{2}{3}\right)} \| Q \|_{M^{\frac{p}{2},\gamma}_{t,x}}\, \| \vu \|_{M^{p,\gamma}_{t,x}}\,  T^{\frac{p-3}{p}}. 
\end{split}
\end{equation*}

To control the term $\| Q \|_{M^{\frac{p}{2},\gamma}_{t,x}}$, using the identity (\ref{Q}) and second point of Lemma \ref{Lem-Tech-2}, we obtain
\begin{equation*}
\| Q \|_{M^{\frac{p}{2},\gamma}_{t,x}} =\, \left\| \sum_{i,j=1}^{3} \mathcal{R}_i \mathcal{R}_j (u_i u_j) \right\|_{M^{\frac{p}{2},\gamma}_{t,x}} \leq C \,\| \vu \otimes \vu \|_{M^{\frac{p}{2},\gamma}_{t,x}} \leq \, C\,  \| \vu \|^2_{M^{p,\gamma}_{t,x}}.
\end{equation*}

With this control, we have
\begin{equation}\label{Estim-I3}
I_2 \leq C\, R^{3\left( \frac{2}{3}-\frac{3}{p}+\frac{\gamma}{p}\right)} \, \| \vu\|^3_{M^{p,\gamma}_{t,x}} \, T^{\frac{p-3}{p}}.
\end{equation}

Finally, the desired estimate now follows from estimates (\ref{Estim-I1})  and (\ref{Estim-I3}).
 \end{proof}

{\bf End of the proof of Theorem \ref{Th1}}. Once we dispose of inequalities proven in Lemmas \ref{Lem-Tech-Energy-1} and \ref{Lem-Tech-Energy-2}, returning to estimate (\ref{Estim-Base}) we find that 
\begin{equation*}
\begin{split}
&\, \int_{B_{\frac{R}{2}}} |\vu(t, \cdot)|^2 \,dx  + 2\int_{0}^{t}\int_{B_{\frac{R}{2}}}  |\vn \otimes \vu(s,\cdot) |^2 \, dxds \\
\leq & \| \vu_0 \|^2_{L^2} + C\, T^{\frac{p-2}{p}}\, R^{-\frac{1}{3}} \, \| \vu \|^2_{M^{p,\gamma}_{t,x}}+ C\, T^{\frac{p-3}{p}}\, R^{3\left(\frac{\gamma}{p}-\frac{3}{p}+\frac{2}{3}\right)} \| \vu \|^3_{M^{p,\gamma}_{t,x}}.
\end{split}
\end{equation*}	 

Consequently, letting $R\to +\infty$ and using the assumption (\ref{Relationship}) we find that 
\[ \lim_{R \to +\infty} \left(  C\, T^{\frac{p-2}{p}}\, R^{-\frac{1}{3}} \, \| \vu \|^2_{M^{p,\gamma}_{t,x}}+ C\, T^{\frac{p-3}{p}}\, R^{3\left(\frac{\gamma}{p}-\frac{3}{p}+\frac{2}{3}\right)} \| \vu \|^3_{M^{p,\gamma}_{t,x}}\right)=0, \] hence
 we obtain the desired global energy inequality (\ref{Energy-Inequality}). We thus  conclude that  $\vu \in L^\infty_t L^2_x \cap L^2_t \dot{H}^1_x\big([0,T]\times \Rt\big)$.   Theorem \ref{Th1} is now proven. 

\section{Proof of Corollary \ref{Cor:Uniqueness}}  
 From Theorem \ref{Th1} we known that $\vu \in M^{p,\gamma}_{t,x}\big( [0,T]\times \Rt \big)$, is a Leray solution of equations (\ref{NS}).  Here, recall that the parameters $\gamma$ and $p$ satisfy the relationship given in (\ref{Relationship}), in particular we have  $3\leq p < \frac{3(3-\gamma)}{2}<\frac{9}{2}$. 
 
 \medskip
 
 On the other hand,  from the additional assumption (\ref{Assumption-Uniqueness-Regularity}) and the definition of the space $M^{\delta,q}_{t,x}\big( [0,T_1]\times \Rt \big)$ given in expression (\ref{Local-Morrey-tx}), for any 
bounded domain $\Omega \subset \Rt$ we have that 
\[ \vu \in L^q \big( [0,T_1], L^q(\Omega) \big). \]
Now, recall that the parameter $q$ satisfies the relationship given in (\ref{Assumption-Uniqueness-Regularity}). Then, as $p<\frac{9}{2}$ we have $q\geq \frac{3p}{p-2}>p$.

 \medskip
 
 Therefore, we can write
 \begin{equation}\label{Ladyzhenskaya-Prodi-Serrin-Uniq-Reg}
 \vu \in L^q \big( [0,T_1], L^q(\Omega) \big) \subset L^p\big( [0,T_1], L^{\frac{3p}{p-2}}(\Omega) \big),
 \end{equation}
 where the parameters $p$ and $\frac{3p}{p-2}$ satisfy the Ladyzhenskaya-Prodi-Serrin relationship given in (\ref{Ladyzhenskaya - Prodi - Serrin criterion}). 
 
 \medskip

 Using this information and the fact that $\vu$ is a Leray solution of equations (\ref{NS}), in particular it satisfies the global energy inequality (\ref{Energy-Inequality}),   from \cite[Theorem $4.2$]{Galdi-2000} we obtain that $\vu$ is the unique Leray solution associated with $\vu$ on $[0,T_1]\times \Omega$. Finally, as the bounded domain $\Omega$ is arbitrary, we can deduce that $\vu$ is the unique Leray solution on $[0,T_1]\times \Rt$.  Corollary \ref{Cor:Uniqueness} is thus proven.

 \section{Proof of Corollary \ref{Cor:Regularity}} Following the same arguments above, we also have that the Leray solution $\vu$ satisfies (\ref{Ladyzhenskaya-Prodi-Serrin-Uniq-Reg}). In this case, we choose $\Omega$ as a bounded domain uniformly  of class $\mathcal{C}^{\infty}$. 
 
 \medskip
 
 Then, from \cite[Theoem $5.2$]{Galdi-2000} we obtain that $\vu \in \mathcal{C}^{\infty}\big(]0,T_1] \times \overline{\Omega}\big)$.  As before, as the regular domain $\Omega$ is arbitrary, we deduce that $\vu \in \mathcal{C}^{\infty}\big( ]0,T_1] \times \Rt \big)$. Corollary \ref{Cor:Regularity} is now proven.

\section{Appendix}\label{AppendixA}  
\subsection{Proof of embedding (\ref{Embedding1})} It is sufficient to prove the last embedding. Let $R\geq 1$.  For $q<r\leq +\infty$ from \cite[Proposition $1.1.10$]{Chamorro-book} we have 
\[ \int_{|x|<R} | \phi(t,x)|^p\, dx \leq C\, R^{3\left(1-\frac{p}{q} \right)}\, \| \phi(t,\cdot)\|^p_{L^{q,\infty}} \leq C\,  R^{3\left(1-\frac{p}{q} \right)}\, \| \phi(t,\cdot)\|^p_{L^{q,r}}, \]
hence,  we write
\[ \frac{1}{R^\gamma} \int_{|x|<R} | \phi(t,x)|^p\, dx \leq C\,  R^{3\left(1-\frac{p}{q} \right)-\gamma}\, \| \phi(t,\cdot)\|^p_{L^{q,r}}. \]
In the last estimate we need that $3\left(1-\frac{p}{q} \right)-\gamma \leq 0$. This fact is  equivalent to the inequality
\[ \frac{2}{3}-\frac{3}{q}\leq \frac{\gamma}{p}-\frac{3}{p}+\frac{2}{3}, \]
 which is finally ensured by the relation given in (\ref{Relationship}). 
\begin{Remark}
	Note that $\frac{2}{3}-\frac{3}{q}<0$ imposes the constraint $q<\frac{9}{2}$. 
\end{Remark}	

Once we have  $3\left(1-\frac{p}{q} \right)-\gamma \leq 0$, we return to the previous estimate.   As $R\geq 1$ and integrating on the interval of time $[0,T]$, we obtain 
\[ \frac{1}{R^\gamma} \int_{0}^{T} \int_{|x|<R} | \phi(t,x)|^p\, dx dt  \leq C\,  \int_{0}^{T}  \| \phi(t,\cdot)\|^p_{L^{q,r}} \, dt. \]
\subsection{Proof of embedding (\ref{Embedding2})} For $p<q$ and the parameter $\gamma$, we assume that $\gamma> 3 \left( 1-\frac{p}{q}\right)$. 
\begin{Remark}
	As $\gamma$ and $p$ always satisfy the relationship (\ref{Relationship}), from this last assumption we have 
	\[ \frac{3}{p}  \left(1- \frac{p}{q} \right) - \frac{3}{p}+\frac{2}{3}< \frac{\gamma}{p}-\frac{3}{p}+\frac{2}{3}<0, \]
	hence we obtain that  $q<\frac{9}{2}$. 
\end{Remark}

Then, for any $R\geq 1$, using the norm $\| \cdot \|_{\dot{M}^{p,q}}$ defined in (\ref{Homogeneous-Morrey}), we can write
\begin{equation*}
\begin{split}
&\, \left( \frac{1}{R^\gamma} \int_{0}^{T} \int_{|x|<R} |\phi(t,x)|^p\, dxdt \right)^{\frac{1}{p}} \leq \, \left( \frac{1}{R^{3\left(1-\frac{p}{q}\right)}} \int_{0}^{T}\int_{|x|<R}|\phi(t,x)|^p\, dx dt \right)^{\frac{1}{p}}\\
\leq &\, \left( \int_{0}^{T} R^{\frac{3p}{q}}\, \frac{1}{R^3} \int_{|x|<R}|\phi(t,x)|^p\, dx dt  \right)^{\frac{1}{p}}\leq \, \left( \int_{0}^{T} \left[ R^{\frac{3}{q}} \left( \frac{1}{R^3} \int_{|x|<R} |\phi(t,x)|^p\, dx \right)^{\frac{1}{p}} \right]^p\, dt  \right)^{\frac{1}{p}}\\
\leq &\, C\, \left( \int_{0}^{T} \| \phi(t,\cdot)\|^p_{\dot{M}^{p,q}}\, dt  \right)^{\frac{1}{p}}. 
\end{split}
\end{equation*}
\subsection{Proof of Lemma \ref{Lem-Tech-Parabolic-Morrey}}. We follow similar ideas as above.  Now, for $p<q$ we must assume that $\gamma> 5 \left( 1- \frac{p}{q}\right)$. 
\begin{Remark}
	As before as $\gamma,p$ satisfy (\ref{Relationship}) we have 
\[ \frac{5}{p}\left( 1- \frac{p}{q} \right)- \frac{3}{p}+\frac{2}{3}< \frac{\gamma}{p}-\frac{3}{p}+\frac{2}{3}<0,  \]
and then
\[ \frac{2}{p}+\frac{2}{3}<\frac{5}{q} \]
$\frac{2}{p}+\frac{2}{3}<\frac{5}{q}$. Additionally, the condition $p<q$ implies that  
\[  \frac{2}{q}+ \frac{2}{3}< \frac{2}{p}+ \frac{2}{3}< \frac{5}{q}, \]
hence we have the constrain  $q<\frac{9}{2}$. 
  \end{Remark}	

Let $R\geq 1$ and $0<T \leq 2$. Observe that this last condition implies $\sqrt{\frac{T}{2}}\leq 1 \leq R$, which we will use later to justify some computations. 

\medskip

Now, for $R$ and $T$ given,  let $t_0\in \R$ such that $T-R^2 \leq t_0 \leq R^2$.  Note that, on the one hand,  from these estimates we have $[0,T]\subset [t_0-R^2, t_0+R^2]$. On the other hand, from these estimates we also have $T\leq 2R^2$, hence $\sqrt{\frac{T}{2}}\leq R$. Nevertheless,  this fact does not imposes any constraint on $R\geq 1$ since, as previously mentioned, we have $\sqrt{\frac{T}{2}}\leq 1 \leq R$. 

\medskip

Then, for any $R\geq 1$,  using the norm $\| \cdot \|_{\mathcal{M}^{p,q}_{t,x}}$ defined in expression (\ref{Parabollic-Morrey}),  we write
\begin{equation*}
\begin{split}
&\, \left( \frac{1}{R^\gamma} \int_{0}^{T}\int_{|x|<R} |\phi(t,x)|^p\, dx dt  \right)^{\frac{1}{p}} \leq \, \left( \frac{1}{R^{5\left(1-\frac{p}{q}\right)}} \int_{0}^{T}\int_{|x|<R}|\phi(t,x)|^p\, dxdt \right)^{\frac{1}{p}}\\
\leq&\, \left( \frac{1}{R^{5\left(1-\frac{p}{q}\right)}} \int_{t_0-R^2}^{t_0+R^2} \int_{|x|<R} \left|\mathds{1}_{[0,T]}(t) \phi(t,x)\right|^p\, dxdt  \right)^{\frac{1}{p}} \leq C\, \left\| \mathds{1}_{[0,T]}(\cdot) \phi \right\|_{\mathcal{M}^{p,q}_{t,x}}. 
\end{split}
\end{equation*}

Lemma \ref{Lem-Tech-Parabolic-Morrey} is proven. 

\medskip

\paragraph{{\bf Statements and Declaration}}
Data sharing does not apply to this article, as no datasets were generated or analyzed during the current study. In addition, the author declares that he has no conflicts of interest.

\vspace{5mm} 

\end{document}